\documentclass[final,1p,times,number]{elsarticle}

\usepackage{amssymb}
\usepackage{amsthm}

\usepackage{tikz}

\usepackage{algorithmic}

\theoremstyle{definition}
\newtheorem{de}{Definition}[section]
\newtheorem{rem}[de]{Remark}
\newtheorem{ex}[de]{Example}

\theoremstyle{plain}
\newtheorem{thm}[de]{Theorem}
\newtheorem{prop}[de]{Proposition}

\newcommand{\cL}{\cal{L}}

\newcommand{\rank}{\mathop{\mathrm{rank}}\nolimits}

\newcommand{\Q}{\mathbb Q}

\newcommand{\Z}{\mathbb Z}
\newcommand{\N}{\mathbb N}

\newcommand{\lcm}{{\rm lcm}}
\newcommand{\ord}{{\rm ord}}
\newcommand{\ld}{{\rm ld}}
\newcommand{\init}{{\rm init}}

\newcommand{\diag}{{\rm diag}}

\newcommand{\LL}{{\cal L}}

\newenvironment{algorithm}[1]{
  \begin{center}
    {\bf Computation of #1}\\*
     \begin{tabular}{|p{120mm}|} \hline
} {
 \\ \hline
 \end{tabular}
 \end{center}
}

\newcommand{\labexzerothomasfirst}{($T_1$)}
\newcommand{\labexzerothomassecond}{($T_2$)}
\newcommand{\labexzerothomasthird}{($T_3$)}
\newcommand{\labexzerothomasfourth}{($T_4$)}
\newcommand{\labexzerothomasfifth}{($T_5$)}

\newcommand{\labexonethomasfirst}{($T_1$)}
\newcommand{\labexonethomassecond}{($T_2$)}
\newcommand{\labexonethomasthird}{($T_3$)}
\newcommand{\labexonethomasfourth}{($T_4$)}
\newcommand{\labexonethomasfifth}{($T_5$)}
\newcommand{\labexonethomassixth}{($T_6$)}

\newcommand{\labexfivethomasfirst}{($T_1$)}
\newcommand{\labexfivethomassecond}{($T_2$)}
\newcommand{\labexfivethomasthird}{($T_3$)}

\newcommand{\labexsixthomasfirst}{($T_1$)}
\newcommand{\labexsixthomassecond}{($T_2$)}
\newcommand{\labexsixthomasthird}{($T_3$)}
\newcommand{\labexsixthomasfourth}{($T_4$)}
\newcommand{\labexsixthomasfifth}{($T_5$)}

\journal{Advances in Applied Mathematics}

\begin{document}

\begin{frontmatter}



\title{Lagrangian Constraints and Differential Thomas Decomposition}


\author[vpg]{Vladimir P. Gerdt\corref{cor1}}
\ead{gerdt@jinr.ru}
\cortext[cor1]{Corresponding author}
\address[vpg]{Laboratory of Computing Techniques and Automation, Joint Institute for Nuclear Research,\\ 141980 Dubna, Russia}

\author[dr]{Daniel Robertz}
\ead{daniel.robertz@plymouth.ac.uk}
\address[dr]{Centre for Mathematical Sciences, Plymouth University, 2-5 Kirkby Place, Drake Circus,\\ Plymouth PL4~8AA, UK}

\begin{abstract}
In this paper we show how to compute algorithmically the
full set of algebraically independent constraints
for singular mechanical and field-theoretical models
with polynomial Lagrangians. If a model under consideration is not
singular as a whole but has domains of dynamical (field) variables where
its Lagrangian becomes singular, then our approach allows to detect
such domains and compute the relevant constraints. In doing so, we
assume that the Lagrangian of a model is a differential polynomial and apply the
differential Thomas decomposition algorithm to the
Euler-Lagrange equations.
\end{abstract}


\end{frontmatter}


%

\section{Introduction}\label{sec:introduction}

Models with singular Lagrangians play a fundamental role in modern elementary particle physics and quantum field theory~(cf.~\citep{Wipf,Rothe}).
Singularity of such models is caused by local symmetries of their Lagrangians. Gauge symmetry is the most important type of local symmetries and it is imperative for all physical theories of fundamental interactions. The local symmetry transformations of differential equations describing dynamical systems or field-theoretical models relate their solutions satisfying the same initial (Cauchy) data. For dynamical systems with only one independent variable the initial data include (generalized) coordinates and velocities whereas for field-theoretical models they include the field variables, their spatial and the  first-order temporal derivatives (`velocities'). The presence of local symmetries in a singular model implies that its general solution satisfying the initial data depends on arbitrary functions.

A distinctive feature of singular Lagrangian models is that their dynamics is governed by the Euler-Lagrange equations which have differential consequences in the form of internal (hidden) constraints for the initial data. This is in contrast to regular constrained dynamics whose constraints are external with respect to the Euler-Lagrange equations.

Given a model Lagrangian, it is very important to verify whether it is singular,
and if so to compute internal constraints that are hidden in the Euler-Lagrange equations.
Knowledge of such constraints is necessary in constrained dynamics for the local symmetry analysis~\citep{Wipf} and for quantization~\citep{Rothe} of the singular model. As it was observed first in~\citep{SeilerCD1,SeilerCD2} in the framework of the Dirac formalism for investigation and quantization of constrained dynamical systems and gauge field theories, the Dirac procedure of the constraint calculation is nothing else than completion of the Hamiltonian equation system to involution. In this respect, the Hamiltonian constraints are \emph{integrability conditions}~\citep{Seiler}. The algorithmic aspects of computation and separation of Hamiltonian constraints into the first and second classes were studied in~\citep{GG}. Computation of independent Lagrangian constraints as integrability conditions for the Euler-Lagrange equations is essential for well-posedness of initial value problems~\citep{Seiler,Finikov,Deriglazov}. Such constraints are detected in the course of completion of the Euler-Lagrange equations to involution.

In the present paper we consider Lagrangian models whose Lagrangians (mechanics) and
Lagrangian densities (field theory) are differential polynomials. Under this
condition we show that the Thomas decomposition~\citep{BGLHR,Robertz}, being a
characteristic one (cf.~\citep{Hubert}) for the radical differential ideal generated
by the polynomials in Euler-Lagrange equations, provides an algorithmic tool for the
computation of Lagrangian constraints. Unlike the traditional linear algebra based
methodology used in theoretical and mathematical physics~\citep{Wipf} for
computing linearly independent Lagrangian constraints, our approach is not only
fully algorithmic but also takes into account rank dependence of the Hessian matrix
on the dynamical (field) variables and outputs the complete set of algebraically
independent constraints. The Thomas decomposition splits the system of
Euler-Lagrange equations into a set of involutive subsystems and automatically
determines the complete set of constraints for each output subsystem which has a
certain rank deficiency of the Hessian. Moreover, each output subsystem admits a
formally well-posed Cauchy problem~\citep{Seiler,Finikov,Gerdt10} in terms of the
output constraints. In contrast to other types of characteristic decompositions
of (radical) differential ideals, a Thomas decomposition consists of differential
systems whose solution sets are pairwise disjoint. This is accomplished by the
use of inequations (cf.\ Example~\ref{ex:notdisjoint}), allowing an irredundant
description of the cases exhibiting different sets of constraints.

This paper is organized as follows. Section~\ref{sec:preliminaries} contains a
short description of our research objects. In Section~\ref{sec:linearalgebra}
we outline the standard approach~\citep{Wipf} to the computation of Lagrangian
constraints. Our approach to this problem is based on the general algorithmic technique
of differential Thomas decomposition~\citep{BGLHR,Robertz} outlined in
Section~\ref{sec:thomasdecomposition}. In Section~\ref{sec:application} we adopt
this technique to the application to constrained dynamics and give a rigorous algebraic description of this adaptation.
In Section~\ref{sec:examples} we illustrate its application by a number of examples including both field-theoretic and dynamical system models. Some concluding remarks are given in Section~\ref{sec:conclusion}.

\section{Preliminaries}\label{sec:preliminaries}

In the framework of their field-theoretical description
most of the fundamental laws of physics can be understood via the \emph{action}~\citep{Wipf,Rothe}
\[
S=\int dt \int d^nx \, {\cal{L}}(\varphi^a,\partial_{t}\varphi^a,\partial_{x_1}\varphi^a,\ldots,\partial_{x_n}\varphi^a)\,,
\]
where ${\cL}$ is a Lagrangian density depending on the \emph{field variables}
\[
\varphi^a=\varphi^a(x_1,\ldots,x_n,t)\,, \quad a\in \{1,\ldots,m\}\, ,
\]
and their first-order partial derivatives. In terms of its density ${\cL}$
the \emph{Lagrangian} $L$ is expressed as
\[ 	
 L=\int d^nx \, {\cal{L}}(\varphi^a,\partial_{t}\varphi^a,\partial_{x_i}\varphi^a,\ldots,\partial_{x_n}\varphi^a)\,.
\]
The principle of least action (Hamilton's variational principle) implies the (field) Euler-Lagrange equations
\begin{equation}\label{field}
\partial_{x_\mu} \frac{\partial {\cL}}{\partial (\partial_{x_\mu} \varphi^a)}-\frac{\partial {\cL}}{\partial \varphi^a}=0,\,\ a\in \{1,\ldots,m\}\,, \ \mu\in \{0,\ldots,n\}\,,\ \ x_0\equiv t\,. 	
\end{equation}

Here and in what follows we use Einstein's summation convention, i.e.,
summation over all possible values of the repeated indices occurring
in a single term is assumed.

For dynamical systems the action and the Euler-Lagrange equations read~\citep{Wipf,Rothe,Deriglazov}
\begin{equation}\label{mechanics}
S=\int dt \, L(q^a,q^a_t)\,,\quad  \frac{d}{dt} \left(\frac{\partial L}{\partial q^a_t}\right) - \frac{\partial L}{\partial q^a}=0\,,\quad a\in \{1,\ldots,m\}\,,
\end{equation}
where $q^a$ are (generalized) \emph{coordinates} and $q^a_t\equiv {dq^a}/dt$ are their \emph{velocities}\,.

A Lagrangian is \emph{regular} if its Hessian matrix or Hessian defined as
\begin{equation}
H_{i,j}=\left\lbrace
\begin{array}{l}
\displaystyle
\frac{\partial^2 L}{{\partial q^i_t} \, {\partial q^j_t}}\quad \mbox{(Dynamical System)}\,, \\[0.5cm]
\displaystyle
\frac{\partial^2 {\cL}}{{\partial \varphi^i_t} \, {\partial \varphi^j_t}}\quad \mbox{(Field-Theoretic Model)}\,,
\end{array}
\right. \label{hessian}
\end{equation}
is invertible, and \emph{singular} otherwise.

\begin{rem}
If a Lagrangian $L$ possesses a continuous group of local gauge transformations, then its Hessian (\ref{hessian}) has a non-trivial nullspace. This fact follows from the generalized Bianchi identity (cf.~\citep{Wipf}). Thus, gauge invariant theories are unavoidably singular.
\end{rem}

Let $K$ be a differential field of characteristic zero. By definition, $K$ is a field containing
the rational numbers $\Q$ with commuting derivations $\delta_0$, $\delta_1$, \ldots,
$\delta_n$, i.e., additive maps $\delta_i: K \to K$ satisfying the Leibniz rule
\[
\delta_i(k_1 \, k_2) = \delta_i(k_1) \, k_2 + k_1 \, \delta_i(k_2) \quad
\mbox{for all } k_1, k_2 \in K, \quad
\mbox{and} \quad \delta_i \circ \delta_j = \delta_j \circ \delta_i \quad
\mbox{for all } i, j.
\]
Moreover, let $R_n = K\{ \varphi^1, \ldots, \varphi^m \}$ be
the \emph{partial differential polynomial ring} with
differential indeterminates $\varphi^1$, \ldots, $\varphi^m$,
endowed with the set $\Delta=\{\partial_t, \partial_{x_1},
\partial_{x_2},\ldots,\partial_{x_n}\}$ of commuting derivations.
In other words, $R_n$ is the polynomial ring
\[
K[\partial_t^i \partial_{x_1}^{j_1} \ldots \partial_{x_n}^{j_n} \varphi^k \mid i, j_1, \ldots, j_n \in \Z_{\ge 0}, k \in \{ 1, \ldots, m \}]
\]
with infinitely many indeterminates, where the action of
$\partial_t$, $\partial_{x_1}$, $\partial_{x_2}$, \ldots, $\partial_{x_n}$
on an indeterminate increases the corresponding exponent $i$, $j_1$, \ldots, $j_n$
by one and on $K$ coincides with the action of $\delta_0$, $\delta_1$, \ldots, $\delta_n$.
We identify $\varphi^i$
with $\partial_t^0 \partial_{x_1}^0 \ldots \partial_{x_n}^0 \varphi^i$ and
also indicate differentiation by subscripts; e.g.,
$\varphi^1_{t,t,x_1} = \partial_t^2 \partial_{x_1} \varphi^1$.
We shall denote by $R_0$ the \emph{ordinary differential polynomial ring} with
differential indeterminates $\varphi^1$, \ldots, $\varphi^m$
and $\Delta=\{\partial_t\}$.

In our study of the Euler-Lagrange equation systems (\ref{field}) and (\ref{mechanics}) we shall restrict  our consideration to field-theoretical and mechanical models such that
\begin{equation}
 {\cL}\in R_n\ \mbox{and}\ L\in R_0\,,\ \mbox{respectively}.\label{dif_pol_ring}
\end{equation}
It should be noted that this restriction is consistent with most of fundamental physical field theories~\citep{Wipf,Rothe} whose (singular) Lagrangian densities are differential polynomials in the field variables and their first-order partial derivatives over a field of constants.

\section{Linear algebra based approach}\label{sec:linearalgebra}

In terms of the Hessian (\ref{hessian}) the set of Euler-Lagrange equations in (\ref{field}) and (\ref{mechanics}) can be written as
\begin{equation}
E:=\{\,e_i=0\mid i=1,\ldots,m\,\}\,,\quad
e_i:=\left\lbrace
\begin{array}{l}
H_{i,j} \, q^j_{tt}+P_i\quad \mbox{(Dynamical System)}\,, \\[0.3cm]
H_{i,j} \, \varphi^j_{tt}+P_i\quad \mbox{(Field-Theoretic Model)}\,,
\end{array}
\right. \label{el-equations}
\end{equation}
where the differential polynomials $H_{i,j}$ and $P_i$ contain derivations in $t$ of order at most $1$.

If the Hessian $H$ is singular, then there are differential consequences of equations in (\ref{el-equations}) that do not contain the second order derivatives in $t$.
These consequences, called \emph{constraints}, are of principal importance in the
study of singular mechanical and field-theoretical models
(cf.\ \citep{Wipf}--\citep{Deriglazov}). Given a Lagrangian, the standard approach
used in physics and mechanics to compute Lagrangian constraints (cf.~\citep{Wipf})
can be formulated as the following procedure:

\begin{algorithm}{constraints based on linear algebra \label{LC}}
\begin{enumerate}[{Step} 1]
\item Compute the Hessian $H$ in accordance to~(\ref{hessian}). Then derive the set $E$ of Euler-Lagrange equations~(\ref{el-equations}) of cardinality $m:=|E|$ and let $C:=\{\,\}$\,.

\item Compute the rank $r$ of the Hessian taking into account equations in $E$.

\item If $r=m$, then go to Step~6. Otherwise, go to the next step.

\item Compute a basis $V$ of the nullspace of $H$, set up
\[
C:=\{\,P_iV^i_\alpha \mid \alpha=1,\ldots,|V|\,\}
 \]
and enlarge the equation set
\[
E:=E\cup \{\,c=0 \mid c\in C\setminus \{0\}\,\}\,.
\]
\item Set $m:=r$ and go to Step~2.

\item Return $C$.
\end{enumerate}
\end{algorithm}

The elements in $C$ that have differential order at most $1$ are called \emph{Lagrangian}~(cf., for example, \citep{Wipf,Krupkova}),
and the set of \emph{Lagrangian constraints} is formed by a largest functionally independent subset.
\vskip 0.5cm
The above procedure suffers from the following algorithmic drawbacks.
\begin{enumerate}
\item \label{drawback1}
The entries of the Hessian are ordinary or partial differential polynomials, i.e.,
we have $H_{i,j}\in R_0$ or $H_{i,j}\in R_n$ (cf.\ Section~\ref{sec:preliminaries}).
If Gaussian elimination is applied in Steps~2 and 4, the computation is
performed in the field of fractions of $R_0$ or $R_n$ in general. Division by
a non-zero differential polynomial which vanishes on the solution set of the
given system has to be prevented. Such differential polynomials are not detected
by the above procedure (cf.\ also \ref{drawback3} below).

On the other hand, the rings $R_0$ and $R_n$ are not principal ideal domains.
For this reason the Hessian $H$ does not admit a Smith normal form in general,
which would provide an algorithmic way to compute $\rank(H)$ in Step~2.
In the case of a multivariate polynomial matrix its Smith normal form can be
defined and computed only in very special cases (cf.~\citep{Boudellioua} and
the references therein) that are not related to the Hessian of the
general form (\ref{hessian}). Similar remarks apply to the computation of a
basis of the nullspace in Step~4 (cf., for example,~\citep{Zhou} and its bibliography).
\item \label{drawback2}
Generally, and we illustrate this fact by examples in Section~\ref{sec:examples},
the rank of the Hessian may vary from one solution subspace of the Euler-Lagrange equations (\ref{field}) or (\ref{mechanics}) to another one. Accordingly, the set of Lagrangian constraints may depend on a solution domain.
\item \label{drawback3}
Computation of $\rank(H)$ in Step~2 assumes simplification (reduction) of the Hessian
modulo (solution of) the system of Euler-Lagrange equations and its extension, respectively. Algorithmically, such a simplification can be performed by means of a characteristic decomposition~\citep{Hubert} of the radical differential ideal generated by the left hand sides of the Euler-Lagrange equations. Thus, the procedure has to be extended with such a kind of decomposition.
\item \label{drawback4}
The output set $C$ of constraints has to be further processed to extract the set
of Lagrangian constraints.
\end{enumerate}

Below we show that differential Thomas decomposition~\citep{BGLHR,Robertz} applied to the Euler-Lagrange equations (\ref{el-equations}) provides a fully algorithmic way to compute an algebraically independent set of Lagrangian constraints that takes into account the dependence of the rank of the Hessian on the solution domain.

\section{Thomas decomposition and simple systems}\label{sec:thomasdecomposition}

We sketch the idea of Thomas decomposition of systems of polynomial ordinary or
partial differential equations (ODEs or PDEs) and its
computation. This method traces back to work by J.~M.\ Thomas in the
1930s~\citep{Thomas}. For more details, we also refer to \citep{BGLHR,Robertz}.

Let
\begin{equation}\label{eq:originalPDE}
p_1 = 0, \quad \ldots, \quad p_s = 0, \quad
q_1 \neq 0, \quad \ldots, \quad q_h \neq 0
\qquad (s, h \in \Z_{\ge 0})
\end{equation}
be a system of polynomial ordinary or partial differential equations
and inequations, i.e., the left hand sides $p_1$, \ldots, $p_s$,
$q_1$, \ldots, $q_h$ are elements of the differential polynomial ring
$R_0$ or $R_n$ over the differential field $K$ of characteristic zero.

A fundamental problem is to determine all power series solutions
of (\ref{eq:originalPDE}). Writing the unknown functions $\varphi^1$,
\ldots, $\varphi^m$ (of either $t$ or of $t$, $x_1$, \ldots, $x_n$) as
power series with indeterminate coefficients, we are led to substitute
this ansatz into (\ref{eq:originalPDE}) and compare coefficients.
However, this naive procedure does not take into account conditions on
the Taylor coefficients of $\varphi^1$, \ldots, $\varphi^m$ which
arise as non-trivial consequences by differentiating the equations
in (\ref{eq:originalPDE}), taking linear combinations of their left
hand sides, etc. Therefore, system (\ref{eq:originalPDE}) needs to be
transformed into an equivalent system which incorporates all
integrability conditions and which is in this sense \emph{formally integrable}~(cf.~\citep{Seiler}).
For nonlinear systems of PDEs certain case distinctions are necessary
in general, resulting in a finite family of formally integrable PDE
systems (e.g., a differential Thomas decomposition) such that the
union of their solution sets equals the solution set of the original
PDE system. Differential Thomas decomposition has various applications,
allowing, e.g., to decide membership to the radical differential ideal
which is generated by the given PDE system or to solve certain
differential elimination problems \citep{Robertz}.

In order to identify the most significant variable in each non-constant
(differential) polynomial, a total ordering $\succ$ on the set of (partial)
derivatives of all orders of all unknown functions $\varphi^1$, \ldots,
$\varphi^m$ is fixed beforehand. It is assumed to respect
differentiation and that differentiation increases the significance
of terms. The ordering $\succ$ is referred to as \emph{ranking}.

\medskip

We distinguish two stages of the computation of a differential
Thomas decomposition, which may be intertwined in practice: its
\emph{algebraic part} and its \emph{differential part}. The
\emph{algebraic part} of Thomas' algorithm eliminates variables
(representing unknown functions and their derivatives) in equations
and inequations by applying polynomial division and ensures square-freeness
of all left hand sides. For these reductions
the corresponding left hand sides are considered as univariate
polynomials in their most significant variables, which are commonly
referred to as their \emph{leaders}, with coefficients that are polynomials
in lower ranked variables.

More precisely, if two distinct equations $p_i = 0$ and $p_j = 0$
have the same leader $v$ and $\deg_v(p_i) \ge \deg_v(p_j)$, then there
exist (differential) polynomials $c_1$, $c_2$, and $r$ such that
\begin{equation}\label{eq:pseudodiv}
c_1 \, p_i - c_2 \, p_j = r
\end{equation}
and $r$ is zero or its leader is ranked lower than $v$ or its
leader is $v$ and $\deg_v(r) < \deg_v(p_j)$. This pseudo-division
eliminates the highest power of $v$ in $p_i$, where $c_1$ may be chosen
as a suitable power of the \emph{initial} $\init(p_j)$ of $p_j$, i.e.,
the coefficient of the highest power of $v$ in $p_j$. Replacing the
equation $p_i = 0$ with the equation $r = 0$ in the system requires
that $c_1$ does not vanish on the solution set of the system, if the
solution set of the system is to be maintained by this operation. For this
reason, Thomas' algorithm possibly splits the system into two systems which
are subsequently treated in the same way as the original system, which is
discarded. These two systems are defined by inserting the inequation
$\init(p_j) \neq 0$ and the equation $\init(p_j) = 0$, respectively.

For each two distinct inequations $q_i \neq 0$ and $q_j \neq 0$ with
the same leader the least common multiple (lcm) $r$ of $q_i$ and $q_j$ is
computed and the two inequations are replaced with $r \neq 0$. This
process involves pseudo-divisions as above and also distinguishes
cases of vanishing or non-vanishing initials.

For each pair $p_i = 0$, $q_j \neq 0$ of an equation and an inequation
with the same leader the greatest common divisor $r$ of $p_i$ and $q_j$
is computed, where polynomials are considered again as univariate
polynomials in their leaders and it is required that initials of
pseudo-divisors do not vanish on the solution set of the system.
If $p_i$ divides $q_j$, then the system is inconsistent and is discarded.
If $r \in K \setminus \{ 0 \}$, then the inequation $q_j \neq 0$ is
removed from the system. Otherwise, the equation $p_i = 0$ is replaced
with the equation $p_i / r = 0$.

Finally, further reductions and splittings of systems may be necessary
to ensure square-freeness of all left hand sides (as univariate polynomials),
i.e., non-vanishing of their \emph{discriminants}. Since $K$
is of characteristic zero, square-freeness of a univariate polynomial
is characterized by the condition that the polynomial and its derivative
have no root in common. Adhering to the recursive representation of
polynomials, a computation of the greatest common divisor of a left
hand side and its derivative with case distinctions as above allows to
determine the square-free part.

\medskip

The above procedure is a variant of Euclid's algorithm for
multivariate polynomials, terminating after finitely many steps. If the
fact that the variables occurring in $p_i$ and $q_j$ represent unknown
functions and their derivatives is neglected, then the result is a
finite collection of algebraic systems with the following property.

\begin{de}\label{de:algsimple}
An algebraic system $S$ as in (\ref{eq:originalPDE}) is said to be
\emph{simple} (with respect to $\succ$), if the following three conditions
are satisfied, where $S_{\prec v}$ is the subsystem of $S$
consisting of those equations and inequations whose leader is ranked lower
than the variable $v$.
\begin{enumerate}
\item \label{algcond:a}
All $p_i$ and all $q_j$ are non-constant polynomials.
\item \label{algcond:b}
The leaders of all $p_i = 0$ and $q_j \neq 0$ are pairwise distinct.
\item \label{algcond:c}
If $v$ is the leader of $p_i = 0$ or $q_j \neq 0$, then neither the
initial nor the discriminant of that equation or inequation has a solution
(over the complex numbers) in common with the subsystem $S_{\prec v}$.
\end{enumerate}
\end{de}

\begin{rem}
A simple algebraic system $S$ can in principle be solved recursively over
the complex numbers by extending the coordinates of a solution of $S_{\prec v}$
with a root of the equation with leader $v$, if present, or with a
complex number satisfying the inequation with leader $v$, or with an
arbitrary complex number if $v$ is not the leader of any equation or
inequation in $S$.

Moreover, conditions \ref{algcond:a} and \ref{algcond:b} imply that
the left hand sides $p_1$, \ldots, $p_s$, $q_1$, \ldots, $q_h$ in a
simple algebraic system are \emph{algebraically independent over $K$},
i.e., the zero polynomial is the only polynomial with $s+h$ indeterminates
and coefficients in $K$ which vanishes under substitution of $p_1$, \ldots,
$p_s$, $q_1$, \ldots, $q_h$.
\end{rem}

The \emph{differential part} of Thomas' algorithm works towards
formal integrability. For each two distinct equations $p_i = 0$ and $p_j = 0$
whose leaders
$\partial^k \varphi^a := \partial_t^{k_0} \partial_{x_1}^{k_1} \ldots \partial_{x_n}^{k_n} \, \varphi^a$
and
$\partial^l \varphi^a := \partial_t^{l_0} \partial_{x_1}^{l_1} \ldots \partial_{x_n}^{l_n} \, \varphi^a$
are derivatives of the same unknown function $\varphi^a$, the algorithm
considers the derivatives of $p_i$ and $p_j$ which both have
leader $\lcm(\partial^k, \partial^l) \, \varphi^a$. In an appropriate linear
combination of these two derivatives the highest power of this leader is
eliminated. Such a linear combination is referred to as a \emph{cross-derivative}
of $p_i$ and $p_j$. The algorithm checks whether pseudo-division of this
cross-derivative modulo (derivatives of) the equations of the system yields the
zero polynomial. Non-zero remainders resulting from this reduction are added as
new equations to the system, and the augmented system is treated by the
algebraic part of Thomas' algorithm again. A method which traces back to
M.~Janet \citep{Janet} allows an efficient organization of all cross-derivatives.
By distinguishing admissible and non-admissible derivations for each equation,
each cross-derivative is considered only once.

The differential part performs pseudo-divisions as in (\ref{eq:pseudodiv}),
which however involve derivatives of equations of the system in general.
If the equation $p_i = 0$ is replaced with the equation $r = 0$, then we
require that $c_1$ does not vanish on the solution set of the system. If a
proper derivative of $p_j$ is subtracted, $c_1$ may be chosen as a suitable
power of the initial of that derivative, which is the \emph{separant} of $p_j$,
namely, the partial derivative of $p_j$ with respect to its leader. Since
the discriminant of a differential polynomial (with respect to its leader)
is essentially the resultant of the polynomial and its separant, the
algebraic part of Thomas' algorithm ensures that separants do not vanish on
the solution set of the system.

The result of Thomas' algorithm is a finite collection of differential
systems with the following property.

\begin{de}\label{de:diffsimple}
A differential system $S$ as in (\ref{eq:originalPDE}) is said to be
\emph{simple} (with respect to $\succ$),
if the following three conditions are satisfied.
\begin{enumerate}
\item The system $S$ is simple as an algebraic system (with respect
to $\succ$, cf.\ Def.~\ref{de:algsimple}).
\item The cross-derivative of each pair of distinct equations whose leaders
involve the same unknown function reduces to zero modulo the equations of
the system and their derivatives.
\item The left hand side $q_j$ of every inequation is reduced
modulo the left hand sides of the equations, in the sense that
no pseudo-division of $q_j$ modulo any $p_i$ is possible.
\end{enumerate}
\end{de}

\begin{rem}
Note that each splitting of a system $S$ into two systems $S_1$ and $S_2$
performed by Thomas' algorithm is arranged so as to ensure that the
solution set of $S$ is the disjoint union of the solution sets of $S_1$
and $S_2$.
\end{rem}

\begin{de}\label{de:thomasdecomp}
Let $S$ be a differential system as in (\ref{eq:originalPDE}). A
finite collection $S_1$, \ldots, $S_k$ of simple differential systems
whose solution sets form a partition of the solution set of $S$
is called a \emph{differential Thomas decomposition} of $S$ (with
respect to $\succ$).
\end{de}

A Thomas decomposition is not uniquely determined and its computation
depends on the choice of the ranking $\succ$ and many other choices.
Among the possible improvements of the above outline of Thomas' algorithm
is the option to factorize the left hand sides of equations and to split
a system according to a non-trivial factorization.

The algorithmic details of the Thomas decomposition method have been
worked out in \citep{BGLHR} (cf.\ also \citep{Robertz}). Implementations
have been developed by T.~B{\"a}chler and M.~Lange-Hegermann
as Maple packages {\tt AlgebraicThomas} and {\tt DifferentialThomas}
and are freely available \citep{Packages}.

\medskip

We illustrate the Thomas decomposition technique on a simple example.

\begin{ex}\label{ex:notdisjoint}
Let us consider the PDE given by the following Hessian determinant:
\begin{equation}\label{eq:hessian}
\det \left( \begin{array}{ccc}
u_{x,x} & u_{x,y} & u_{x,z} \\
u_{x,y} & u_{y,y} & u_{y,z} \\
u_{x,z} & u_{y,z} & u_{z,z}
\end{array} \right) = 0\,.
\end{equation}
The Thomas decomposition method is applied over the differential polynomial ring
$K\{ u \}$ with one differential indeterminate $u$ and
commuting derivations $\partial_x$, $\partial_y$, $\partial_z$,
where $K$ is a differential field of characteristic zero, e.g., $K = \Q$ so that
the restrictions of $\partial_x$, $\partial_y$, $\partial_z$ to $K$ are all zero.
We choose the degree-reverse lexicographical ranking satisfying
$\partial_x \, u \succ \partial_y \, u \succ \partial_z \, u$, i.e.,
\[
\partial_x^{a_1} \partial_y^{a_2} \partial_z^{a_3} \, u \succ
\partial_x^{b_1} \partial_y^{b_2} \partial_z^{b_3} \, u \quad
\iff \quad
\left\{ \begin{array}{l}
a_1 + a_2 + a_3 > b_1 + b_2 + b_3 \quad \mbox{or}\\[0.5em]
( \, a_1 + a_2 + a_3 = b_1 + b_2 + b_3 \quad \mbox{and}\\[0.5em]
(a_1, a_2, a_3) \neq (b_1, b_2, b_3) \quad \mbox{and}\\[0.5em]
a_i < b_i \quad \mbox{for} \quad i = \max \, \{ j \mid a_j \neq b_j \} \, )\,.
\end{array} \right.
\]
Then the initial of the PDE (\ref{eq:hessian}) is
$u_{z,z} \, u_{y,y} - u_{y,z}^2$. We distinguish the cases whether
this initial vanishes or not. In the first case we obtain the following
simple differential system under the assumption that the initial $u_{z,z}$
of the initial of (\ref{eq:hessian}) does not vanish (no integrability
condition needs to be checked for the single equation):
\[
\mbox{\labexzerothomasfirst}\qquad
\left\{
\begin{array}{rcl}
\left( u_{z,z} \, u_{y,y} - u_{y,z}^2 \right) \underline{u_{x,x}} +
2 \, u_{y,z} \, u_{x,z} \, u_{x,y} - u_{z,z} \, u_{x,y}^2 - u_{y,y} \, u_{x,z}^2 & = & 0\,,\\[0.15em]
u_{z,z} \, \underline{u_{y,y}} - u_{y,z}^2 & \neq & 0\,,\\[0.15em]
u_{z,z} & \neq & 0\,.
\end{array}
\right.
\]
Leaders of differential polynomials are underlined where not obvious.
If $u_{z,z}$ vanishes in the first case, a simplification yields the second
simple differential system
\[
\mbox{\labexzerothomassecond}\qquad
\left\{
\begin{array}{rcl}
-u_{y,z}^2 \, \underline{u_{x,x}} + 2 \, u_{y,z} \, u_{x,z} \, u_{x,y} - u_{y,y} \, u_{x,z}^2 & = & 0\,,\\[0.15em]
u_{y,z} & \neq & 0\,,\\[0.15em]
u_{z,z} & = & 0
\end{array}
\right.
\]
after pseudo-reduction of a single cross-derivative to zero modulo the system.
If the initial of (\ref{eq:hessian}) vanishes, we again distinguish the
cases whether $u_{z,z}$ vanishes or not. If not, the pseudo-division
\[
u_{z,z} \left( 2 \, u_{y,z} \, u_{x,z} \, u_{y,z} - u_{z,z} \, u_{x,y}^2 -
u_{y,y} \, u_{x,z}^2 \right) + u_{x,z}^2 \left( u_{z,z} \, u_{y,y} - u_{y,z}^2 \right) =
\left( u_{z,z} \, u_{x,y} - u_{y,z} u_{x,z} \right)^2
\]
allows to replace the PDE (\ref{eq:hessian}) with the expression in the
bracket on the right hand side (ensuring square-freeness).
We obtain the simple differential system
\[
\mbox{\labexzerothomasthird}\qquad
\left\{
\begin{array}{rcl}
u_{z,z} \, \underline{u_{x,y}} - u_{y,z} \, u_{x,z} & = & 0\,,\\[0.15em]
u_{z,z} \, \underline{u_{y,y}} - u_{y,z}^2 & = & 0\,,\\[0.15em]
u_{z,z} & \neq & 0
\end{array}
\right.
\]
after confirming that the integrability condition for the two equations
is satisfied by applying pseudo-reduction modulo the system.
Finally, if the initial of the PDE (\ref{eq:hessian}) vanishes, then
$u_{z,z} = 0$ implies $u_{y,z} = 0$, and (\ref{eq:hessian}) reduces to
$u_{y,y} \, u_{x,z}^2 = 0$. Distinguishing the cases whether $u_{y,y}$
vanishes or not, we obtain the simple differential systems
\[
\mbox{\labexzerothomasfourth}\qquad
\left\{
\begin{array}{rcl}
u_{x,z} & = & 0\,,\\[0.15em]
u_{y,z} & = & 0\,,\\[0.15em]
u_{z,z} & = & 0\,,\\[0.15em]
u_{y,y} & \neq & 0
\end{array}
\right.
\]
and
\[
\mbox{\labexzerothomasfifth}\qquad
\left\{
\begin{array}{rcl}
u_{y,y} & = & 0\,,\\[0.15em]
u_{y,z} & = & 0\,,\\[0.15em]
u_{z,z} & = & 0\,.
\end{array}
\right.
\]
Thus, we have the following tree of case distinctions (suppressing branches leading to inconsistent systems):
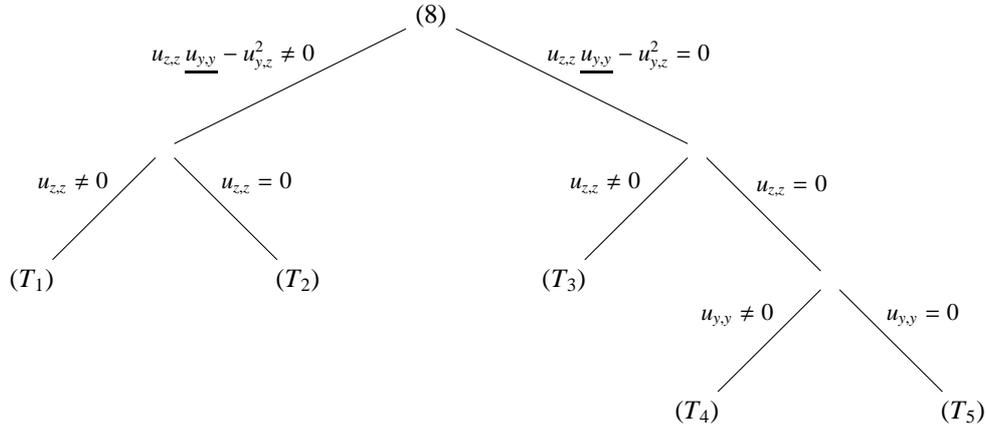
\begin{center}
\begin{tikzpicture}[node distance=1.75cm]
\node (PDE) {(\ref{eq:hessian})};
\node (TX) [below of=PDE] {};
\node (TA) [left of=TX] {};
\node (TC) [left of=TA] {};
\node (TD) [below of=TC] {};
\node (T1) [left of=TD] {\labexzerothomasfirst};
\node (T2) [right of=TD] {\labexzerothomassecond};
\node (TB) [right of=TX] {};
\node (TU) [right of=TB] {};
\node (TV) [below of=TU] {};
\node (T3) [left of=TV] {\labexzerothomasthird};
\node (TW) [right of=TV] {};
\node (TY) [below of=TW] {};
\node (T4) [left of=TY] {\labexzerothomasfourth};
\node (T5) [right of=TY] {\labexzerothomasfifth};
\draw (PDE)  --  (TC) node [near start,left]
{\small $u_{z,z} \, \underline{u_{y,y}} - u_{y,z}^2 \neq 0 \phantom{xx}$};
\draw (PDE)  --  (TU) node [near start,right]
{\small $\phantom{xx} u_{z,z} \, \underline{u_{y,y}} - u_{y,z}^2 = 0$};
\draw (TU)  --  (T3) node [near start,left]
{\small $u_{z,z} \neq 0 \phantom{x}$};
\draw (TU)  --  (TW) node [near start,right]
{\small $\phantom{x} u_{z,z} = 0$};
\draw (TC)  --  (T1) node [near start,left]
{\small $u_{z,z} \neq 0 \phantom{x}$};
\draw (TC)  --  (T2) node [near start,right]
{\small $\phantom{x} u_{z,z} = 0$};
\draw (TW)  --  (T4) node [near start,left]
{\small $u_{y,y} \neq 0 \phantom{x}$};
\draw (TW)  --  (T5) node [near start,right]
{\small $\phantom{x} u_{y,y} = 0$};
\end{tikzpicture}
\end{center}
As opposed to the Thomas decomposition \labexonethomasfirst--\labexonethomasfifth\
with pairwise disjoint solution sets, e.g., the package {\tt DifferentialAlgebra}
in Maple~17 computes, using the Rosenfeld-Gr{\"o}bner algorithm, the following characteristic decomposition for (\ref{eq:hessian}):
\[
\left\{ \begin{array}{rcl}
\left( u_{z,z} \, u_{y,y} - u_{y,z}^2 \right) \underline{u_{x,x}} +
2 \, u_{y,z} \, u_{x,z} \, u_{x,y} - u_{z,z} \, u_{x,y}^2 - u_{y,y} \, u_{x,z}^2 & = & 0\,,\\[0.15em]
u_{z,z} \, \underline{u_{y,y}} - u_{y,z}^2 & \neq & 0
\end{array} \right.
\]
\[
\left\{ \begin{array}{rcl}
u_{z,z} \, \underline{u_{x,y}} - u_{y,z} \, u_{x,z} & = & 0\,,\\[0.15em]
u_{z,z} \, \underline{u_{y,y}} - u_{y,z}^2 & = & 0\,,\\[0.15em]
u_{z,z} & \neq & 0
\end{array} \right.
\qquad\qquad \left( \, = \mbox{\labexzerothomasthird} \, \right)
\]
\[
\left\{ \begin{array}{rcl}
u_{x,z} & = & 0\,,\\[0.15em]
u_{y,z} & = & 0\,,\\[0.15em]
u_{z,z} & = & 0
\end{array} \right.
\]
\[
\left\{ \begin{array}{rcl}
u_{y,y} & = & 0\,,\\[0.15em]
u_{y,z} & = & 0\,,\\[0.15em]
u_{z,z} & = & 0\,.
\end{array} \right.
\qquad\qquad \left( \, = \mbox{\labexzerothomasfifth} \, \right)
\]
The solution sets of these differential systems are not disjoint.
The third and fourth systems have solutions
$f_1(x) + f_2(x) \, y + c \, z$ in common, where
$f_1$ and $f_2$ are arbitrary analytic functions
and $c$ is an arbitrary constant.
\end{ex}

\begin{rem}
Due to its disjointness the Thomas decomposition partitions the solution space of the Euler-Lagrange equations (\ref{field}) or (\ref{mechanics}). In doing so, it may occur that the field-theoretic or mechanical model, which is globally non-singular, is singular in some domains of the solution space. In those domains there are hidden Lagrangian constraints that are revealed by the Thomas decomposition (cf.\ Section~\ref{sec:examples} for examples). In terms of these \emph{local} constraints one can pose an initial value (Cauchy) problem providing existence and uniqueness of local analytic solutions \citep{Seiler,Finikov,Gerdt10}. The local analyticity provides a smooth dependence of the solution on the initial data and thereby makes the initial-value problem well-posed~(cf.~\citep{Hadamard}). It should be noted that the well-posedness of the initial value problem has been proven for an involutive simple differential system~\citep{Finikov} related to the domain under consideration.
For another type of characteristic decomposition which is not disjoint a proof of the
well-posedness for Cauchy problem is not known in the literature.
\end{rem}

\begin{rem}
In some applications one might be interested only in one branch of the tree
of case distinctions. For instance, if the differential ideal which is generated
by the left hand sides of the given system of PDEs is prime, the resulting
Thomas decomposition contains a simple system which is most generic in a precise
sense (cf.\ \cite[Subsection~2.2.3]{Robertz}). Clearly, if only the generic simple
system is required, the computation can be restricted to that branch.
\end{rem}

\section{Adaptation of Thomas decomposition to constrained dynamics}\label{sec:application}

Hereafter, we shall write the underlying ring of differential polynomials in~(\ref{dif_pol_ring})
without subscript, since it will be clear from the context what ring is meant, partial or ordinary. We fix a ranking $\succ$ on $R$ which satisfies
\begin{equation}\label{eq:conditionranking}
\partial_t \, \varphi^a \succ \partial_{x_1}^{k_1}\cdots
\partial_{x_n}^{k_n} \, \varphi^b \quad
\mbox{for all} \quad k_1, \ldots, k_n \in \Z_{\ge 0}, \quad
a, b \in \{ 1, \ldots, m \}\,.
\end{equation}
The leader of any non-constant differential polynomial $p$ in $R$,
i.e., the differential indeterminate or
derivative $\partial_t^j \partial_{x_1}^{k_1} \cdots \partial_{x_n}^{k_n} \, \varphi^a$ which is maximal with respect to $\succ$ among those occurring in $p$,
is denoted by $\ld(p)$.

\begin{rem}\label{rem:elimranking}
Condition (\ref{eq:conditionranking}) implies that
for all $i,j$, $k_1,\ldots,k_n,l_1,\ldots,l_n\in \Z_{\ge 0}$ and for all
$a, b \in \{ 1, \ldots, m \}$ we have
\[
j > k \quad \Rightarrow \quad
\partial_t^j \partial_{x_1}^{k_1} \cdots \partial_{x_n}^{k_n} \,
\varphi^a \succ \partial_t^k \partial_{x_1}^{l_1} \cdots \partial_{x_n}^{l_n} \, \varphi^b.
\]
The claim follows from applying $\partial_t^k$ to the left hand side
and the right hand side of
\[
\partial_t^{j-k} \partial_{x_1}^{k_1} \cdots \partial_{x_n}^{k_n} \,
\varphi^a \succ \partial_t^{j-k} \, \varphi^a \succ
\partial_{x_1}^{l_1} \cdots \partial_{x_n}^{l_n} \, \varphi^b.
\]
\end{rem}

\begin{de}
Let $\psi \in \{ \varphi^1, \ldots, \varphi^m \}$ be minimal with respect to $\succ$.
A non-constant differential polynomial $p \in R$ is said to be
a \emph{$\partial_t^2 \psi$-condition} if
\[
\ld(p) \prec \partial_t^2 \psi
\]
holds.
\end{de}

\begin{de}\label{de:local}
Given Euler-Lagrange equations (\ref{field}) or (\ref{mechanics}),
their differential consequence $p=0$ is called \emph{generalized Lagrangian constraint}
if $p$ is a $\partial_t^2 \psi$-condition. If a constraint holds for all solutions of (\ref{field}) or (\ref{mechanics}) it is called \emph{global}, and \emph{local} otherwise. A non-singular Lagrangian system with local constraints is \emph{locally singular}.
\end{de}

\begin{rem}
For singular models of mechanics, generalized Lagrangian constraints are just Lagrangian ones. For singular (or locally singular) field-theoretical models, Lagrangian constraints, as we underlined in Section~\ref{sec:linearalgebra}, being of the first differential order, form a subset of the set of generalized  Lagrangian (or local Lagrangian) constraints.
\end{rem}

We consider the subring
\[
\begin{array}{rcl}
R_{\ord(t) < 2} & := & K[\partial_t^i \partial_{x_1}^{j_1}
\cdots \partial_{x_n}^{j_n} \varphi^k \mid i \in \{ 0, 1 \},
j_1, \ldots, j_n \in \Z_{\ge 0}, k \in \{ 1, \ldots, m \}] \\[0.5em]
& = & K[\varphi^k, \partial_{x_1} \varphi^k, \partial_{x_2} \varphi^k, \ldots,
\partial_t \varphi^k, \partial_t \partial_{x_1} \varphi^k, \ldots \mid
k \in \{ 1, \ldots, m \}]
\end{array}
\]
of $R$ as a differential ring with derivations
$\partial_{x_1}$, \ldots, $\partial_{x_n}$ and
the ranking induced by $\succ$.

\begin{rem}
Every $\partial_t^2 \psi$-condition is an element of $R_{\ord(t) < 2}$.
\end{rem}

\begin{rem}
Let $I$ be a differential ideal of $R$, i.e., an ideal of $R$ which is
closed under all derivations of $R$.
The set of $\partial_t^2 \psi$-conditions in $I$ is closed under
taking $R_{\ord(t) < 2}$-linear combinations and applying an arbitrary number of
the derivations $\partial_{x_1}, \ldots,\partial_{x_n}$
(cf.\ Remark~\ref{rem:elimranking}). In other words, the set of
$\partial_t^2 \psi$-conditions in $I$ is a differential ideal
of $R_{\ord(t) < 2}$.
\end{rem}

We denote by $S^{=}$ the set of left hand sides of equations in a
differential system $S$. The \emph{saturation} of a differential ideal $E$ of $R$
by a differential polynomial $q \in R$ is defined as the differential ideal
$E : q^{\infty} := \{ \, p \in R \mid q^k \, p \in E \mbox{ for some }
k \in \Z_{\ge 0} \, \}$.

\begin{prop}\label{prop:cond}
Let $S$ be a simple differential system over $R$ with respect to a
ranking satisfying (\ref{eq:conditionranking}) and let
$\psi \in \{ \varphi^1, \ldots, \varphi^m \}$ be minimal with respect to $\succ$.
Let $E$ be the differential ideal of $R$ which is generated by $S^=$
and let $q$ be the product of the initials and separants of all elements of $S^=$.
Moreover, let
\[
S_{\ord(t) < 2} := \{ \, p \in S^= \mid \ld(p) \prec \partial_t^2 \psi \, \}\,,
\]
$E_{\ord(t) < 2}$ the differential ideal of $R_{\ord(t) < 2}$ which is
generated by $S_{\ord(t) < 2}$, and $q_{\ord(t) < 2}$ the product of the
initials and separants of all elements of $S_{\ord(t) < 2}$.
Then the set of $\partial_t^2 \psi$-conditions in $I := E : q^{\infty}$ is
the radical differential ideal $E_{\ord(t) < 2} : (q_{\ord(t) < 2})^{\infty}$
of $R_{\ord(t) < 2}$.
\end{prop}

\begin{proof}
Let $p$ be a $\partial_t^2 \psi$-condition in $I$. Since $S$ is a simple
differential system, repeated differential pseudo-reduction of $p$ modulo $S^=$
yields zero (cf. \citep[Prop.~2.2.50]{Robertz}).
Since we have $\ld(p) \prec \partial_t^2 \psi$, only elements of
$S_{\ord(t) < 2}$ are chosen as pseudo-divisors. Their initials and separants
are elements of $R_{\ord(t) < 2}$ and the coefficients which are used for
the pseudo-reductions are in $R_{\ord(t) < 2}$ as well. This shows that the
set of $\partial_t^2 \psi$-conditions in $I$ is contained
in $E_{\ord(t) < 2} : (q_{\ord(t) < 2})^{\infty}$. Conversely, every element of
$E_{\ord(t) < 2} : (q_{\ord(t) < 2})^{\infty}$ is clearly a
$\partial_t^2 \psi$-condition in $I$.
\end{proof}

We recall that the radical of a differential ideal $I$ of $R$ is defined as
the differential ideal
\[
\sqrt{I} := \{ \, p \in R \mid p^k \in I \mbox{ for some } k \in \N \, \}\,.
\]

\begin{thm}\label{thm:dtcond}
Let
\begin{equation}\label{eq:PDEsystem}
p_1 = 0, \quad \ldots, \quad p_s = 0, \quad
q_1 \neq 0, \quad \ldots, \quad q_h \neq 0
\end{equation}
be a (not necessarily simple) system of polynomial partial differential equations
and inequations, where $p_1$, \ldots, $p_s$, $q_1$, \ldots, $q_h \in R$,
$s$, $h \in \Z_{\ge 0}$.
We denote by $E$ the differential ideal of $R$ which is generated
by $p_1$, \ldots, $p_s$
and by $T$ the set of $\partial_t^2 \psi$-conditions in
$\sqrt{E : q^{\infty}}$, where $q := q_1 \cdot \ldots \cdot q_h$.
Let
\[
S^{(1)}, \quad \ldots, \quad S^{(r)}
\]
be a differential Thomas decomposition of (\ref{eq:PDEsystem})
with respect to a ranking satisfying (\ref{eq:conditionranking}) and let
$\psi \in \{ \varphi^1, \ldots, \varphi^m \}$ be minimal with respect to $\succ$.
For $i = 1$, \ldots, $r$, let
\[
T^{(i)} := E^{(i)}_{\ord(t) < 2} : (q^{(i)}_{\ord(t) < 2})^{\infty},
\]
where $E^{(i)}_{\ord(t) < 2}$ is the differential ideal of $R_{\ord(t) < 2}$
which is generated by
\[
S^{(i)}_{\ord(t) < 2} := \{ \, p \in (S^{(i)})^= \mid \ld(p) \prec \partial_t^2 \psi \, \}
\]
and $q^{(i)}_{\ord(t) < 2}$ is the product of the initials and separants of all
elements of $S^{(i)}_{\ord(t) < 2}$. Then we have
\[
T = T^{(1)} \cap \ldots \cap T^{(r)}.
\]
\end{thm}

\begin{proof}
We have
\[
\sqrt{E : q^{\infty}} = (E^{(1)} : (q^{(1)})^{\infty}) \cap \ldots \cap
(E^{(r)} : (q^{(r)})^{\infty}),
\]
where $E^{(i)}$ is the differential ideal of $R$ which is generated by
$(S^{(i)})^=$ and $q^{(i)}$ is the product of the initials and separants
of all elements of $(S^{(i)})^=$, where $i \in \{ 1, \ldots, r \}$
(cf.\ \citep[Prop.~2.2.72]{Robertz}). From Proposition~\ref{prop:cond}
we conclude that,
for each $i \in \{ 1, \ldots, r \}$,
the set of $\partial_t^2 \psi$-conditions in
$E^{(i)} : (q^{(i)})^{\infty}$ is equal to $T^{(i)}$.
\end{proof}

In the examples of the next section it is a very simple task to trace the
origin of the $\partial_t^2 \psi$-conditions arising in the
differential Thomas decomposition.

\section{Examples}\label{sec:examples}

In this section we apply the Thomas decomposition to four models. Two of these models are field-theoretic and the other two are dynamical systems.

\begin{ex}\label{ex:sch}
The first model is the (1+1)-dimensional version of the chiral Schwinger model~\citep{Schwinger,Jackiw}. In the framework of Hamiltonian formalism the last model was studied, for example, in~\citep{Rothe,DasGhosh}.  We consider the Lagrangian density  in the form  (\,\citep{DasGhosh}, Equation~19\,)
\begin{eqnarray*}
& \LL=& \frac{1}{2}\,(\partial_t A_0-\partial_x A_1)^2+\frac{1}{2}\,(\partial_t\phi)^2 -\frac{1}{2}\,(\partial_x\phi)^2 + e\,(\partial_t\phi)\,A_0+e\,\phi(\partial_t\,A_1) \\
&& +\ e\,(\partial_t\phi)\,(A_0-A_1) +\frac{1}{2}\,a\,e^2\,(A_0^2-A_1^2)\,.\\
\end{eqnarray*}
Here, $e$, $a$ are parameters, $t$, $x$ are the independent variables
and $\varphi^1=A_0$, $\varphi^2=A_1$, $\varphi^3=\phi$ are the dependent variables.

Let $\succ$ be the ranking
satisfying (\ref{eq:conditionranking}) and such that
\begin{eqnarray*}
&& \phi \prec A_1 \prec A_0 \prec \partial_x\phi \prec \partial_x A_1 \prec \partial_x A_0 \prec \partial_x^2\phi \prec \ldots \\
&& \prec \partial_t\phi \prec \partial_t A_1 \prec \partial_t A_0 \prec \partial_t\partial_x \phi  \prec \partial_t\partial_x A_1 \prec \partial_t\partial_x A_0
\prec \partial_t\partial^2_x \phi \prec \ldots
\end{eqnarray*}
The Euler-Lagrange equations read
\begin{equation}\label{eq:euler-lagrange}
\left\{ \begin{array}{rcl}
\underline{\partial_t^2A_0}-\partial_t\partial_x A_1 - e\,(\partial_t\phi + \partial_x\phi) - a\,e^2\,A_0 &=& 0\,,\\[0.5em]
\underline{\partial_t\partial_x A_0} -e\,(\partial_t\phi+\partial_x\phi)-\partial_x^2A_1-a\,e^2\,A_1 & = & 0\,, \\[0.5em]
\underline{\partial^2_t\phi}+e\,(\partial_tA_0-\partial_tA_1)-\partial^2_x\phi+e\,(\partial_x A_0
- \partial_x A_1)& = & 0\,.
\end{array} \right.
\end{equation}

The Hessian (\ref{hessian}) of this field-theoretic model is $H=\diag(1,0,1)$ and, hence, the model is singular. The Euler-Lagrange equations (\ref{eq:euler-lagrange}) are linear.
In this case Thomas' algorithm applied to the system performs its completion to involution and outputs the Janet involutive form of (\ref{eq:euler-lagrange}) (cf.~\citep{BGLHR,Robertz}). Denote the first and second equation in (\ref{eq:euler-lagrange}) by $P$ and $Q$, respectively. Then the first step of the completion procedure consists in detection of the (underlined) leaders, calculation of the cross-derivative
\[
\partial_x P-\partial_t Q = 0\,,
\]
and then elimination (reduction) of terms in the left hand side modulo equations in (\ref{eq:euler-lagrange}), as explained in Section~\ref{sec:thomasdecomposition}. As a result of the completion, Thomas' algorithm yields
\begin{equation}\label{eq:involutive}
\left\{ \begin{array}{rcl}
\mbox{\boldmath $(1-a)\,\partial_tA_0+(1+a)\,\partial_x A_0-\partial_t A_1-\partial_x A_1$} &\mbox{\boldmath $=$}&\mbox{\boldmath $0$}\,,\\[0.5em]
(1+a)\,(\partial_t^2 A_1-\partial^2_x A_1)- e\,(2+a)\,(\partial_t\phi+\partial_x\phi) - a\,e^2\,(A_0+A_1)-a^2e^2\,A_1 & = & 0\,, \\[0.5em]
(a+1)(\partial_t\partial_xA_1-\partial^2_xA_0) - e\,(\partial_t\phi+\partial_x\phi)-a\,e^2\,A_1 & = & 0\,,\\[0.5em]
\partial_t^2\phi-\partial^2_x\phi-e\,a\,(\partial_xA_0-\partial_tA_1) & = & 0\,.
\end{array} \right.
\end{equation}
The first equation in (\ref{eq:involutive}) (marked in boldface) is a Lagrangian constraint.
\end{ex}

\medskip
\begin{ex}\label{ex:ex1}
To show a field-theoretic situation when the differential Thomas decomposition computes local (generalized) Lagrangian constraints (cf.\ Definition~\ref{de:local}) for a (globally) non-singular (but locally singular) model, consider the following Lagrangian density:
\begin{equation}\label{ft-model}
\LL =
\frac{1}{2} \, u_t^2 + u \, v_t^2 \, w_x + w_t \, (u_t + v_x)\,.
\end{equation}
The independent variables are $t$, $x$ and the dependent
variables are $\varphi^1 = u$, $\varphi^2 = v$, $\varphi^3 = w$.
We choose the ranking $\succ$
satisfying (\ref{eq:conditionranking}) such that
\[
w \prec v \prec u \prec w_x \prec v_x \prec u_x \prec w_{x,x} \prec \ldots
\prec w_t \prec v_t \prec u_t \prec w_{t,x} \prec v_{t,x} \prec u_{t,x}
\prec w_{t,x,x} \prec \ldots
\]
The density (\ref{ft-model}) generates the Euler-Lagrange equations
\begin{equation}\label{eq:expde1}
\left\{ \begin{array}{rcl}
\underline{u_{t,t}} + w_{t,t} - v_t^2 \, w_x & = & 0\,,\\[0.5em]
2 \, u \, w_x \, \underline{v_{t,t}} + 2 \, u_t \, v_t \, w_x + 2 \, u \, v_t \, w_{t,x} + w_{t,x} & = & 0\,,\\[0.5em]
\underline{u_{t,t}} + v_{t,x} + 2 \, u \, v_t \, v_{t,x} + u_x \, v_t^2 & = & 0\,.
\end{array} \right.
\end{equation}
Let $A$, $B$, $C$ denote the first, second and third equation, respectively.
Considering the (underlined) leaders of these equations,
Thomas' algorithm first replaces $C$ with $C - A$. This corresponds to a row
reduction of the Hessian:
\[
H^{(1)} = \left(
\begin{array}{ccc}
1 & 0 & 1 \\
0 & 2 \, u \, w_x & 0 \\
1 & 0 & 0
\end{array}
\right)\,.
\]
The leaders of the Euler-Lagrange equations
are then $u_{t,t}$, $v_{t,t}$, $w_{t,t}$. Since these differential polynomials involve
pairwise distinct dependent variables, formal integrability of
the system requires no further check (i.e., technically speaking, all derivations
are admissible in the sense of Janet division). Only the possible vanishing of the
initial of the second equation has be taken into account (compare also with the
entry at position $(2,2)$ of the Hessian $H^{(1)}$).

The field model (\ref{eq:expde1}) is not singular. For this reason the
generic case leads to a simple system not containing any constraints
(cf.\ \labexonethomasfirst\ below).
We conclude, using Theorem~\ref{thm:dtcond},
that among the consequences of the given PDE system (\ref{eq:expde1})
there exist no (generalized) Lagrangian constraints.

\medskip

There are precisely two possibilities for the initial of the
second equation to vanish: $w_x = 0$ or $u = 0$.
All local generalized Lagrangian constraints
for these more special consequences
of the given PDE system (\ref{eq:expde1}) arise from the imposed
equations or their \emph{integrability conditions}~\citep{Seiler}.

\medskip

In the first case (i.e., $w_x = 0$), the integrability condition for the
non-zero equations in (\ref{eq:expde1})
\[
\underline{w_{x}} = 0\,, \qquad
v_{t,x} + u_x \, v_t^2 + 2 \, u \, v_t \, v_{t,x} - \underline{w_{t,t}} = 0
\]
is
\[
(2 \, u \, v_t + 1) \, \underline{v_{t,x,x}} + 2 \, u \, v_{t,x}^2 +
4 \, u_x \, v_t \, v_{t,x} + u_{x,x} \, v_t^2 = 0\,.
\]
We obtain three simple systems distinguishing the cases whether
the initial $2 \, u \, v_t + 1$ of the last equation vanishes
(cf.\ \labexonethomasfourth\ below),
or has vanishing initial $u$
(cf.\ \labexonethomasthird\ below),
or has non-vanishing initial $u$
(cf.\ \labexonethomassecond\ below).

\medskip

In the second case (i.e., $u = 0$, $w_x \neq 0$), the integrability condition for
\[
\underline{w_{t,t}} - v_t^2 \, w_x = 0\,, \qquad
\underline{w_{t,x}} = 0
\]
is
\[
2 \, v_t \, w_x \, \underline{v_{t,x}} = 0\,.
\]
We obtain two simple systems distinguishing the cases whether $v_t$ vanishes or not
(cf.\ \labexonethomasfifth\ and \labexonethomassixth\ below, respectively).

\medskip

Summarizing we obtain the following tree of case distinctions
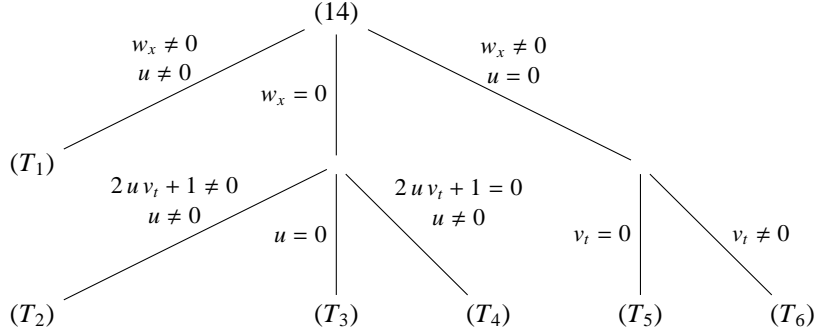
\begin{center}
\begin{tikzpicture}[node distance=2cm]
\node (PDE) {(\ref{eq:expde1})};
\node (TX) [below of=PDE] {};
\node (TY) [right of=TX] {};
\node (TZ) [right of=TY] {};
\node (T3) [below of=TX] {\labexonethomasthird}; 
\node (TW) [left of=T3] {};
\node (T2) [left of=TW] {\labexonethomassecond}; 
\node (T4) [right of=T3] {\labexonethomasfourth}; 
\node (TV) [left of=TX] {};
\node (T1) [left of=TV] {\labexonethomasfirst}; 
\node (T5) [right of=T4] {\labexonethomasfifth}; 
\node (T6) [right of=T5] {\labexonethomassixth}; 
\draw (PDE)  --  (T1) node [near start,left]
{\small $\begin{array}{c} w_x \neq 0 \\ u \neq 0 \end{array} \phantom{xx}$};
\draw (PDE)  --  (TX) node [midway,left]
{\small $w_x = 0$};
\draw (PDE)  --  (TZ) node [near start,right]
{\small $\phantom{xx} \begin{array}{c} w_x \neq 0 \\ u = 0 \end{array}$};
\draw (TX)  --  (T2) node [near start,left]
{\small $\begin{array}{c} 2 \, u \, v_t + 1 \neq 0 \\ u \neq 0 \end{array}$};
\draw (TX) --  (T3) node [midway,left]
{\small $u = 0$};
\draw (TX) --  (T4) node [near start,right]
{\small $\!\begin{array}{c} 2 \, u \, v_t + 1 = 0 \\ u \neq 0 \end{array}$};
\draw (TZ) --  (T5) node [midway,left]
{\small $v_t = 0$};
\draw (TZ) --  (T6) node [midway,right]
{\small $\phantom{x} v_t \neq 0$};
\end{tikzpicture}
\end{center}
where \labexonethomasfirst, \ldots, \labexonethomassixth\
are the simple differential systems of the following Thomas decomposition
of (\ref{eq:expde1}) and where for every simple subsystem we mark its
local generalized Lagrangian constraints in boldface:
\[
\mbox{\labexonethomasfirst}\qquad
\left\{
\begin{array}{rcl}
u_{t,t} + (2 \, u \, v_t + 1) \, v_{t,x} + u_x \, v_{t}^{2} & = & 0\,,\\[0.2em]
2 \, u \, w_{x} \, v_{t,t} + (2 \, u \, v_{t} + 1) \, w_{t,x} + 2 \, w_x \, v_t \, u_t & = & 0\,,\\[0.2em]
(2 \, u \, v_t + 1) \, v_{t,x} - w_{t,t} + (u_x + w_x) \, v_{t}^{2} & = & 0\,,\\[0.2em]
w_{x} & \neq & 0\,,\\[0.2em]
u & \neq & 0
\end{array}
\right.
\]
\[
\mbox{\labexonethomassecond}\qquad
\left\{
\begin{array}{rcl}
u_{t,t} + (2 \, u \, v_t + 1) \, v_{t,x} + u_x \, v_{t}^2 & = & 0\,,\\[0.2em]
\mbox{\boldmath $(2 \, u \, v_t + 1) \, v_{t,x,x} + 2 \, u \, v_{t,x}^2 + 4 \, u_x \, v_t \, v_{t,x}
+ v_{t}^2 \, u_{x,x}$} & \mbox{\boldmath $=$} & \mbox{\boldmath $0$}\,,\\[0.2em]
(2 \, u \, v_t + 1) \, v_{t,x} + u_x \, v_{t}^{2} - w_{t,t} & = & 0\,,\\[0.2em]
\mbox{\boldmath $w_{x}$} & \mbox{\boldmath $=$} & \mbox{\boldmath $0$}\,,\\[0.2em]
u & \neq & 0\,,\\[0.2em]
2 \, u \, v_{t} + 1 & \neq & 0
\end{array}
\right.
\]
\[
\mbox{\labexonethomasthird}\qquad
\left\{
\begin{array}{rcl}
\mbox{\boldmath $u$} & \mbox{\boldmath $=$} & \mbox{\boldmath $0$}\,,\\[0.2em]
\mbox{\boldmath $v_{{t,x}}$} & \mbox{\boldmath $=$} & \mbox{\boldmath $0$}\,,\\[0.2em]
w_{t,t} & = & 0\,,\\[0.2em]
\mbox{\boldmath $w_{x}$} & \mbox{\boldmath $=$} & \mbox{\boldmath $0$}
\end{array}
\right.
\]
\[
\mbox{\labexonethomasfourth}\qquad
\left\{
\begin{array}{rcl}
u_{t,t} & = & 0\,,\\[0.2em]
\mbox{\boldmath $u_{x}$} & \mbox{\boldmath $=$} & \mbox{\boldmath $0$}\,,\\[0.2em]
\mbox{\boldmath $2 \, u \, v_{t}+1$} & \mbox{\boldmath $=$} & \mbox{\boldmath $0$}\,,\\[0.2em]
w_{t,t} & = & 0\,,\\[0.2em]
\mbox{\boldmath $w_{x}$} & \mbox{\boldmath $=$} & \mbox{\boldmath $0$}\,,\\[0.2em]
u & \neq & 0
\end{array}
\right.
\]
\[
\mbox{\labexonethomasfifth}\qquad
\left\{
\begin{array}{rcl}
\mbox{\boldmath $u$} & \mbox{\boldmath $=$} & \mbox{\boldmath $0$}\,,\\[0.2em]
\mbox{\boldmath $v_{t}$} & \mbox{\boldmath $=$} & \mbox{\boldmath $0$}\,,\\[0.2em]
w_{t,t} & = & 0\,,\\[0.2em]
\mbox{\boldmath $w_{t,x}$} & \mbox{\boldmath $=$} & \mbox{\boldmath $0$}\,,\\[0.2em]
w_{x} & \neq & 0
\end{array}
\right.
\]
\[
\mbox{\labexonethomassixth}\qquad
\left\{
\begin{array}{rcl}
\mbox{\boldmath $u$} & \mbox{\boldmath $=$} & \mbox{\boldmath $0$}\,,\\[0.2em]
\mbox{\boldmath $v_{{t,x}}$} & \mbox{\boldmath $=$} & \mbox{\boldmath $0$}\,,\\[0.2em]
w_{t,t} - w_x \, v_{t}^{2} & = & 0\,,\\[0.2em]
\mbox{\boldmath $w_{t,x}$} & \mbox{\boldmath $=$} & \mbox{\boldmath $0$}\,,\\[0.2em]
\mbox{\boldmath $w_{x,x}$} & \mbox{\boldmath $=$} & \mbox{\boldmath $0$}\,,\\[0.2em]
w_{x} & \neq & 0\,,\\[0.2em]
v_{t} & \neq & 0\,.
\end{array}
\right.
\]

Since the differential Thomas decomposition is disjoint (cf.\ Section~\ref{sec:thomasdecomposition}), the simple differential subsystems $(T_1),\ldots,(T_6)$ (cf.\ Definition~\ref{de:diffsimple}) partition the solution space of the Euler-Lagrange equations (\ref{eq:expde1}), and reveal local Lagrangian constraints.
\end{ex}

\medskip
\begin{ex}\label{ex:ex5} As an example of singular dynamical system we consider Lagrangian taken from (\,\citep{Deriglazov}, Equation~8.1\,)
\[
L =
q_2^2 \, (q_1)_t^2 + q_1^2 \, (q_2)_t^2 + 2 \, q_1 \, q_2 \, (q_1)_t \, (q_2)_t + q_1^2 + q_2^2\,.
\]
The independent variable is $t$ and the dependent variables
are $y^1 = q_1$, $y^2 = q_2$.
We choose the ranking $\succ$ such that
\[
q_2 \prec q_1 \prec (q_2)_t \prec (q_1)_t \prec
(q_2)_{t,t} \prec (q_1)_{t,t} \prec \ldots
\]
Then the Euler-Lagrange equations are given by
\begin{equation}\label{eq:expde5}
\left\{ \begin{array}{rcl}
4 \, q_2 \, (q_2)_t \, (q_1)_t + 2 \, q_2^2 \, \underline{(q_1)_{t,t}} + 2 \, q_1 \, q_2 \, (q_2)_{t,t} - 2 \, q_1 & = & 0\,,\\[0.5em]
4 \, q_1 \, (q_2)_t \, (q_1)_t + 2 \, q_1^2 \, (q_2)_{t,t} + 2 \, q_1 \, q_2 \, \underline{(q_1)_{t,t}} - 2 \, q_2 & = & 0\,.
\end{array} \right.
\end{equation}
Let $A$ and $B$ denote the first and second equation, respectively.
According to the (underlined) leaders of these equations, Thomas' algorithm
first performs a pseudo-reduction of $B$ modulo $A$; more precisely, $B$
is replaced with the remainder
\begin{equation}\label{eq:pseudoreductionB}
q_2 \, B - q_1 \, A = 2 \, (q_1^2 - q_2^2) = 0\,.
\end{equation}
This computation corresponds to a row reduction of the Hessian:
\[
H^{(1)} = \left(
\begin{array}{cc}
2 \, q_2^2 & 2 \, q_1 \, q_2 \\
2 \, q_1 \, q_2 & 2 \, q_1^2
\end{array}
\right)\,.
\]
We obtain a zero row
(indeed, the determinant of $H^{(1)}$ vanishes), which is
reflected by the fact that all terms involving differentiation order~$2$ in
(\ref{eq:pseudoreductionB}) cancel. Hence, the model (\ref{eq:expde5}) is singular.

The solution set of the system does not change when
$B$ is replaced with the remainder in (\ref{eq:pseudoreductionB}) if
the coefficient $q_2$ in (\ref{eq:pseudoreductionB}) does
not vanish on the solution set. For this reason, Thomas' algorithm splits
the original system (\ref{eq:expde5}) into one system incorporating the
condition $q_2 \neq 0$, where $B$ is replaced with $q_1^2 - q_2^2 = 0$,
and a complementary system containing the new condition $q_2 = 0$.
By taking the factorization $q_1^2 - q_2^2 = (q_1 - q_2) (q_1 + q_2)$
into account, the first system is split again into two complementary
systems \labexfivethomasfirst\ and \labexfivethomassecond.
For both systems, Thomas' algorithm reduces $A$ modulo $q_1 - q_2 = 0$
or $q_1 + q_2 = 0$, respectively. Both remainders are divided by the
non-vanishing factor $2 \, q_2$.
Note that for ODE systems no formal integrability check is necessary, and
the initials of the resulting left hand sides in
\labexfivethomasfirst\ and \labexfivethomassecond\
do not vanish. By reducing $A$ and $B$ modulo $q_2 = 0$
in the remaining case, we obtain the third simple differential system
\labexfivethomasthird\
of the Thomas decomposition.
The Lagrangian constraints we find are $q_1 - q_2 = 0$ and $q_1 + q_2 = 0$ in
\labexfivethomasfirst\ and \labexfivethomassecond,
respectively (in boldface below).

Thus, we have the following tree of case distinctions
\begin{center}
\begin{tikzpicture}[node distance=2cm]
\node (PDE) {(\ref{eq:expde5})};
\node (TX) [below of=PDE] {};
\node (TW) [left of=TX] {};
\node (T3) [right of=TX] {\labexfivethomasthird}; 
\node (T2) [below of=TX] {\labexfivethomassecond}; 
\node (T0) [left of=T2] {};
\node (T1) [left of=T0] {\labexfivethomasfirst}; 
\draw (PDE)  --  (TW) node [near start,left]
{\small $q_2 \neq 0 \phantom{x}$};
\draw (PDE)  --  (T3) node [near start,right]
{\small $\phantom{x} q_2 = 0$};
\draw (TW)  --  (T1) node [near start,left]
{\small $q_1 - q_2 = 0 \phantom{x}$};
\draw (TW)  --  (T2) node [near start,right]
{\small $\phantom{x} q_1 + q_2 = 0$};
\end{tikzpicture}
\end{center}
where \labexfivethomasfirst, \labexfivethomassecond, \labexfivethomasthird\
are the simple differential systems of the following Thomas decomposition
of (\ref{eq:expde5}):
\[
\mbox{\labexfivethomasfirst}\qquad
\left\{
\begin{array}{rcl}
2 \, q_2 \, (q_2)_{t,t} + 2 \, (q_2)_t^2 - 1 & = & 0\,,\\[0.2em]
\mbox{\boldmath $q_1 - q_2$} & \mbox{\boldmath $=$} & \mbox{\boldmath $0$}\,,\\[0.2em]
q_2 & \neq & 0
\end{array}
\right.
\]
\[
\mbox{\labexfivethomassecond}\qquad
\left\{
\begin{array}{rcl}
2 \, q_2 \, (q_2)_{t,t} + 2 \, (q_2)_t^2 - 1 & = & 0\,,\\[0.2em]
\mbox{\boldmath $q_1 + q_2$} & \mbox{\boldmath $=$} & \mbox{\boldmath $0$}\,,\\[0.2em]
q_2 & \neq & 0
\end{array}
\right.
\]
\[
\mbox{\labexfivethomasthird}\qquad
\left\{
\begin{array}{rcl}
q_1 & = & 0\,,\\[0.2em]
q_2 & = & 0\,.
\end{array}
\right.
\]
Note that the local Lagrangian constraints in the simple systems $(T_1)$ and $(T_2)$ can be combined in a single global constraint \mbox{\boldmath $q_1^2 - q_2^2=0$}\,.
\end{ex}

\medskip
\begin{ex}\label{ex:ex6} The double sombrero model (\,\citep{ZhaoYuXu}, Equation~7\,).
Its Lagrangian is given by
\[
L =
\frac{1}{4} \left( q_1^2 \, (q_2)_t^2 + (q_1)_t^2 - k \right)^2 +
\frac{1}{2} \, \mu \, q_1^2 - \frac{1}{4} \, \lambda \, q_1^4\,.
\]
The independent variable is $t$, the dependent variables
are $y^1 = q_1$, $y^2 = q_2$, and $k$, $\lambda$, and $\mu$ are non-zero parameters.
We choose the ranking $\succ$ such that
\[
q_2 \prec q_1 \prec (q_2)_t \prec (q_1)_t \prec
(q_2)_{t,t} \prec (q_1)_{t,t} \prec \ldots
\]
The Euler-Lagrange equations read
\begin{equation}\label{eq:expde6}
\left\{ \begin{array}{rcl}
2 \left( (q_1)_t \, \underline{(q_1)_{t,t}} + q_1 \, (q_1)_t \, (q_2)_t^2 +
q_1^2 \, (q_2)_t \, (q_2)_{t,t} \right) (q_1)_t \, + & & \\
\left( (q_1)_t^2 + q_1^2 \, (q_2)_t^2 - k \right)
\left( \underline{(q_1)_{t,t}} - q_1 \, (q_2)_t^2 \right) - \mu \, q_1 +
\lambda \, q_1^3 & = & 0\,,\\[0.5em]
2 \left( (q_1)_t \, \underline{(q_1)_{t,t}} + q_1 \, (q_1)_t \, (q_2)_t^2 +
q_1^2 \, (q_2)_t \, (q_2)_{t,t} \right) q_1^2 \, (q_2)_t \, + & & \\
\left( (q_1)_t^2 + q_1^2 \, (q_2)_t^2 - k \right)
\left( 2 \, q_1 \, (q_1)_t \, (q_2)_t + q_1^2 \, (q_2)_{t,t} \right) & = & 0\,.
\end{array} \right.
\end{equation}
The Hessian is
\[
H^{(1)} = \left(
\begin{array}{cc}
q_1^2 \, (q_2)_t^2 + 3 \, (q_1)_t^2 - k &
2 \, q_1^2 \, (q_1)_t \, (q_2)_t \\[0.5em]
2 \, q_1^2 \, (q_1)_t \, (q_2)_t &
2 \, q_1^4 \, (q_2)_t^2 + \left( q_1^2 \, (q_2)_t^2 + (q_1)_t^2 - k \right) q_1^2
\end{array}
\right)
\]
and its determinant is
\[
\det H^{(1)} = q_1^2
\left( q_1^2 \, (q_2)_t^2 + (q_1)_t^2 - k \right)
\left( 3 \, q_1^2 \, (q_2)_t^2 + 3 \, (q_1)_t^2 - k \right).
\]
There are solutions of the Euler-Lagrange equations (\ref{eq:expde6}) such that
the Hessian does not vanish. Therefore, the model is not singular, and there
are no global Lagrangian constraints.

Let $A$ and $B$ denote the first and second equation in (\ref{eq:expde6}),
respectively. We have
\[
\init(A) = q_1^2 \, (q_2)_t^2 + 3 \, (q_1)_t^2 - k\,, \qquad
\init(B) = 2 \, q_1^2 \, (q_1)_t \, (q_2)_t\,.
\]
Under the assumption that $\init(A)$ and its separant $6 \, (q_1)_t$
do not vanish on the
solution set of the system, Thomas' algorithm replaces $B$ with the
equation $\init(A) \, B - \init(B) \, A = 0$, whose leader is $(q_2)_{t,t}$
and whose initial is
\[
q_1 \, \left( q_1^2 \, (q_2)_t^2 + (q_1)_t^2 - k \right)
\left( 3 \, q_1^2 \, (q_2)_t^2 + 3 \, (q_1)_t^2 - k \right).
\]
The coefficient of $(q_2)_{t,t}$ in $A$ is $2 \, q_1^2 \, (q_1)_t \, (q_2)_t$.
Therefore, the algorithm subsequently replaces $A$ with
\[
\left( q_1^2 \, (q_2)_t^2 + (q_1)_t^2 - k \right)
\left( 3 \, q_1^2 \, (q_2)_t^2 + 3 \, (q_1)_t^2 - k \right) A +
2 \, q_1 \, (q_1)_t \, (q_2)_t \, B = 0\,,
\]
whose left hand side is (exactly) divisible by the initial
$q_1^2 \, (q_2)_t^2 + 3 \, (q_1)_t^2 - k$ of the previous equation~$A$.
Under the above assumption, we may divide by this expression and
replace the left hand side with the quotient.
Since the leaders $(q_1)_{t,t}$ and $(q_2)_{t,t}$ of the new equations $A$ and $B$
are derivatives of distinct unknown functions, no integrability check is
necessary. Under the assumption the initials of these equations and
the discriminant of these initials (which is a certain multiple of
$(q_1^2 \, (q_2)_t^2 - k) (3 \, q_1^2 \, (q_2)_t^2 - k)$) do not vanish,
we obtain the simple differential system
\labexsixthomasfirst\
below. The different cases of vanishing initials or discriminants have to be
treated and this yields four other simple differential systems
\labexsixthomassecond--\labexsixthomasfifth.
Note also that two pairs of factors of the inequation in
\labexsixthomasfirst\
become equal if $(q_2)_t$ is specialized to zero, which is the reason for a
case distinction leading to
\labexsixthomassecond.
We obtain the following tree of case distinctions
\begin{center}
\begin{tikzpicture}[node distance=2cm]
\node (PDE) {(\ref{eq:expde6})};
\node (TX) [below of=PDE] {};
\node (TW) [left of=TX] {};
\node (T3) [right of=TX] {\labexsixthomasthird}; 
\node (TA) [below of=TX] {};
\node (T0) [left of=TA] {};
\node (TV) [left of=T0] {};
\node (T2) [below of=TV] {\labexsixthomassecond}; 
\node (T1) [left of=T2] {\labexsixthomasfirst}; 
\node (TB) [right of=T2] {};
\node (T5) [right of=TB] {\labexsixthomasfifth}; 
\node (T4) [right of=T5] {\labexsixthomasfourth}; 
\draw (PDE)  --  (TW) node [near start,left]
{\small $q_1 \neq 0 \phantom{x}$};
\draw (PDE)  --  (T3) node [near start,right]
{\small $\phantom{x} q_1 = 0$};
\draw (TW)  --  (TV) node [near start,left]
{\small $(q_1)_t \neq 0 \phantom{x}$};
\draw (TA)  --  (T5) node [midway,left]
{\small $q_1^2 \, (q_2)_t^2 - k = 0$};
\draw (TA)  --  (T4) node [midway,right]
{\small $\phantom{x} 3 \, q_1^2 \, (q_2)_t^2 - k = 0$};
\draw (TW)  --  (TA) node [near start,right]
{\small $\phantom{x} (q_1)_t = 0$};
\draw (TV)  --  (T1) node [near start,left]
{\small $(q_2)_t \neq 0 \phantom{x}$};
\draw (TV)  --  (T2) node [midway,right]
{\small $(q_2)_t = 0$};
\end{tikzpicture}
\end{center}
where we have only displayed cases leading to consistent systems and where
\labexsixthomasfirst, \ldots, \labexsixthomasfifth\
are the simple differential systems of the following Thomas decomposition
of (\ref{eq:expde6}) with local Lagrangian constraints marked in boldface:
\[
\mbox{\labexsixthomasfirst}\qquad
\left\{
\begin{array}{rcl}
- \left( q_1^2 \, (q_2)_t^2 + (q_1)_t^2 - k \right)
\left( 3 \, q_1^2 \, (q_2)_t^2 + 3 (q_1)_t^2 - k \right)
\underline{(q_1)_{{t,t}}} \, + & & \\[0.15em]
3 \, q_1 \, (q_2)_{{t}}^{2} \, (q_1)_{{t}}^{4} +
q_1 \left( 6\, q_1^2 \, (q_2)_{{t}}^{4}
-4\, k\, (q_2)_{{t}}^{2} -\lambda\, q_1^2
+ \mu \right) (q_1)_{{t}}^{2} \, + & & \\[0.15em]
3 \, q_1^5 \, (q_2)_{{t}}^{6} - 4 \, k \, q_1^{3} \, (q_2)_{{t}}^{4}
- q_1 \left( 3\, \lambda\, q_1^4 - 3\,\mu \, q_1^2 - k^2 \right) (q_2)_{{t}}^{2} \, + & & \\[0.15em]
k \, \lambda\, q_1^{3}-k \, \mu \, q_1 & = & 0\,,\\[0.3em]
\left( q_1^2 \, (q_2)_t^2 + (q_1)_t^2 - k \right)
\left( 3 \, q_1^2 \, (q_2)_t^2 + 3 \, (q_1)_t^2 - k \right)
q_1 \, \underline{(q_2)_{{t,t}}} \, + & & \\[0.15em]
6 \, (q_2)_{{t}} \, (q_1)_{{t}}^{5} +
4 \left( 3\, q_1^{2} \, (q_2)_{{t}}^{2}-2\, k \right) (q_2)_t \, (q_1)_t^{3} \, + & & \\[0.15em]
\left( 6\, q_1^{4} \, (q_2)_{{t}}^{5}-8\, k \, q_1^{2} \, (q_2)_{{t}}^{3}
-2 \left( \lambda \, q_1^{4} - \mu \, q_1^{2} - {k}^{2} \right) (q_2)_{{t}}
 \right) (q_1)_{{t}} & = & 0\,,\\[0.3em]
q_1 \, (q_2)_t \left( q_1^2 \, (q_2)_t^2-k \right)
\left( q_1^2 \, (q_2)_t^2 + (q_1)_t^2 - k \right) \\[0.15em]
\left( q_1^2 \, (q_2)_t^2 + 3 \, (q_1)_t^2 - k \right)
\left( 3 \, q_1^2 \, (q_2)_t^2 - k \right)
\left( 3 \, q_1^2 \, (q_2)_t^2 + 3 \, (q_1)_t^2 - k \right) & \neq & 0\,.
\end{array}
\right.
\]
Note that the initials of the left hand sides of the above two equations
divide $\det H^{(1)}$.
\[
\mbox{\labexsixthomassecond}\qquad
\left\{
\begin{array}{rcl}
\left( 3 \, (q_1)_t^2 - k \right) (q_1)_{t,t} +
\lambda \, q_1^3 - \mu \, q_1 & = & 0\,,\\[0.2em]
\mbox{\boldmath $(q_2)_t$} & \mbox{\boldmath $=$} & \mbox{\boldmath $0$}\,,\\[0.2em]
q_1 \left( (q_1)_t^2-k \right) \left( 3 \, (q_1)_t^2-k \right) & \neq & 0
\end{array}
\right.
\]
\[
\mbox{\labexsixthomasthird}\qquad
\Big\{
\begin{array}{rcl}
\mbox{\boldmath $q_1$} & \mbox{\boldmath $=$} & \mbox{\boldmath $0$}
\end{array}
\]
\[
\mbox{\labexsixthomasfourth}\qquad
\left\{
\begin{array}{rcl}
\mbox{\boldmath $9 \, \lambda \, q_1^4-9 \, \mu \, q_1^2+2 \, k^2$}
& \mbox{\boldmath $=$} & \mbox{\boldmath $0$}\,,\\[0.2em]
\mbox{\boldmath $3 \, q_1^2 \, (q_2)_t^2 - k$}
& \mbox{\boldmath $=$} & \mbox{\boldmath $0$}
\end{array}
\right.
\]
\[
\mbox{\labexsixthomasfifth}\qquad
\left\{
\begin{array}{rcl}
\mbox{\boldmath $\lambda \, q_1^2 - \mu$} & \mbox{\boldmath $=$} & \mbox{\boldmath $0$}\,,\\[0.2em]
\mbox{\boldmath $\mu \, (q_2)_t^2 - k \, \lambda$} & \mbox{\boldmath $=$} & \mbox{\boldmath $0$\,.}
\end{array}
\right.
\]
\end{ex}

\begin{rem}
We note that in our examples Thomas' algorithm is applied to
a system of nonlinear differential equations which are linear in their leaders, i.e.,
in their highest derivatives with respect to the chosen ranking $\succ$.
Therefore, one of the first steps in a computation of a differential Thomas
decomposition for the given PDE system corresponds to the row
reduction of the corresponding Hessian, which is part of the approach based
on linear algebra recalled in Section~\ref{sec:linearalgebra}.
\end{rem}

\section{Conclusion}\label{sec:conclusion}

In this paper we have shown that differential Thomas decomposition applied to the Euler-Lagrange equations of field-theoretical (\ref{field}) or mechanical (\ref{mechanics}) models with polynomial La\-grangi\-ans is a proper algorithmic tool, implemented in Maple~\citep{Packages}, for computing algebraically independent Lagrangian constraints or generalized Lagrangian constraints, respectively.  In doing so, the decomposition being disjoint generates a partition of the solution space of the Euler-Lagrange equations. In the case of a singular or locally singular model the complete set of the corresponding local and algebraically independent constraints is computed for every element in the partition.

One can also apply another characteristic decomposition of the radical differential ideal generated by the Euler-Lagrange equations, for example, that based on the Rosenfeld-Gr\"{o}bner algorithm and implemented  in the {\tt diffalg} library\footnote{Starting from Maple 14, the {\tt diffalg} library was redesigned and renamed to {\tt DifferentialAlgebra}.}\,of Maple (cf.~\citep{Hubert} and its bibliography). Maple has one more built-in splitting algorithm~\citep{Reid} for nonlinear systems of differential equations (command {\tt rifsimp}, a part of the package {\tt DEtools}), that can also be used for the detection of Lagrangian constraints. However, {\tt diffalg} and {\tt rifsimp} do not yield disjoint decompositions in general, and the Lagrangian constraints that are inherent to different output subsystems may interfere.

The ranking (\ref{eq:conditionranking}) is a Riquier ranking~\citep{Riquier,Rust} (cf.\ also~\citep{Seiler}). Thus, by the Riquier existence theorems~\citep{Thomas1,Rust1}, a simple differential subsystem with nonempty equation set provides the existence and uniqueness of formal power series solution  satisfying certain initial (Cauchy) data (cf.~\citep{Seiler,Gerdt10}). If the simple system under consideration is singular, and hence contains (generalized) Lagrangian constraints, then their presence is to be taken into account by the initial data, and the second-order derivatives in time (`accelerations') are uniquely defined.

Differential Thomas decomposition can also be applied to detect and compute hidden constraints in differential-algebraic (and also in partial differential-algebraic) equations (DAEs). The presence of hidden algebraic constraints, i.e., those algebraic constraints that are not explicitly given in the system, is the main obstacle in numerical solving of DAEs (cf., for example,~\citep{KunkelMehrmann}).

\section{Acknowledgements}

The contribution of the first author (V.P.G.) was partially supported by grant No.13-01-00668 from the Russian Foundation for Basic Research.

\bigskip




\bibliographystyle{elsarticle-num-names}
\bibliography{LagrConstr}





\end{document}